\documentclass[11pt]{article}

\usepackage{graphicx}
\usepackage{url}
\usepackage{amsmath}
\usepackage{amsfonts}
\usepackage{verbatim}
\usepackage{amssymb}
\usepackage{amsthm}
\usepackage{color}
\usepackage{mathtools}
\usepackage[affil-it]{authblk}
\usepackage{tikz-cd}

\newtheorem{thm}{Theorem}
\newtheorem{defn}[thm]{Definition}
\newtheorem{prop}[thm]{Proposition}
\newtheorem{cor}[thm]{Corollary}
\newtheorem{lem}[thm]{Lemma}
\newtheorem{remark}[thm]{Remark}

\newtheorem{ex}[thm]{Example}

\def\mcP{{\mathcal{P}}}

\def\g{{\gamma}}
\def\a{{\alpha}}
\def\be{{\beta}}

\newcommand{\bit}{\begin{itemize}}
\newcommand{\eit}{\end{itemize}}
\newcommand{\ben}{\begin{enumerate}}
\newcommand{\een}{\end{enumerate}}
\newcommand{\beq}{\begin{equation}}
\newcommand{\eeq}{\end{equation}}
\newcommand{\bea}{\begin{eqnarray*}}
\newcommand{\eea}{\end{eqnarray*}}
\newcommand{\bpf}{\begin{proof}}
\newcommand{\epf}{\end{proof}\ms}
\newcommand{\ms}{\medskip}

\title{($\varepsilon,T$)-Chains and Chain Recurrence in Graph Determined Hybrid Dynamical Systems }
\author{Kimberly Ayers}
\affil{Pomona College}

\date{}

\begin{document}
\maketitle

\begin{abstract}
This paper examines a continuous time dynamical system that is an extension of a discrete time dynamical system previously examined, and considers this system together in a product space with a compact subset of Euclidean space. Together, the two systems give a skew product flow.  We first examine limit behavior and recurrence in our continuous time extension.  We then consider analogous limit and recurrence concepts for a skew product flow, and the behavior on the Euclidean space that results. 
\end{abstract}

\section{Introduction}
\vspace{2mm}

\indent In this paper, we investigate the concept of a finite number of different flows on the same compact metric space, $M$.  $M$ is often taken to be a subset of $\mathbb{R}^n$ endowed with the usual Euclidean metric, but this is not a necessary requirement - indeed, any compact metric space will do.  We examine the behavior of the system if we transition through the different systems at regular intervals over time.  Additionally, we require that the ``switching" between systems follow rules given by a (not necessarily symmetric) adjacency matrix, or equivalently, a directed graph.  Because the different flows on $M$ are not necessarily related in any way, the limit behavior of this system is not obvious.  We attempted to reconcile this problem of ``combining" limit behavior - what happens when a space has multiple different dynamical systems, each with distinct limit sets, acting on it?  How can limit and recurrence - and in particular, chain recurrence - be studied in this context? We begin by demonstrating that $M$ paired with a function space denoted by $\Delta$ together form a skew-product flow, allowing for the examination of certain limiting behavior and recurrence concepts. \\
\indent In \cite{discretesystems}, we studied a discrete dynamical system on a space $\Omega$ which consists of all bi-infinite paths on a directed graph $G$, endowed with a metric.   This space paired with a flow $\varphi$ given by the left shift-mapping (see \cite{Katok}, pp.48) form a discrete dynamical system.  It can then be shown that there exists a finest Morse Decomposition that is nicely correlated with $G$'s structure, and that the Morse sets of this finest decomposition are either a single periodic orbit or chaotic sets.  This system is a generalization of the behavior seen in Smale's horseshoe (see 
\cite{Robinson}, pp. 275-280).  It is this system $(\Omega, \varphi)$ that we extend to a continuous dynamical system below to line up with the behavior in $M$ to form a skew-product flow.\\
\indent We begin by introducing the function space $\Delta$, our continuous extension of $\Omega$.  We demonstrate that $\Delta$ contains discretization that is topologically conjugate to $\Omega$.  Thus, $\Delta$ inherits many properties from $\Omega$, including an analogous finest Morse Decomposition, and similar chaotic behavior within Morse sets.  We then examine the skew-product flow on $M\times\Delta$, as well as behavior on $M$ considered alone.  It is important to note that $M$ paired with its behavior do not constitute a dynamical systems - the flow property is not satisfied.  Thus, we can not apply many important theorems or definitions to this space.  We thus define some analogous chain recurrence concepts, and explore what the analogous limiting behavior and recurrent sets are in this context.

\section{Continuous Extension of $\Omega$}

Note that in {\rm \cite{discretesystems}}, the flow on $\Omega$ is a discrete time dynamical system.  However, in order to insert this behavior into another dynamical system to create a hybrid system, it is necessary to extend this system to a continuous time dynamical system.  The obvious extension of a sequence into a function on $\mathbb{R}$ is a piecewise constant function.

\begin{defn}\label{delta}
Let 
\begin{center}
$\bar{\Delta} =  \left\{ x : \mathbb{R} \rightarrow V \left|
\begin{array}{cc}
\{x(ih)\}_{i \in \mathbb{Z}} \in \Omega \\
x$ is piecewise constant on $[nh,(n+1)h) \, \forall \, n \in \mathbb{Z}
\end{array}  \right\} \right.$
\end{center}
and 
$$\Delta = \{ x(\cdot + t) | x \in \Delta, t \in \mathbb{R}\} .$$
\end{defn}

\noindent
In other words, $\Delta$ is the sets of functions that results from translating the functions in $\bar\Delta$ by all $t \in \mathbb{R}$.  We allow for all horizontal translations of functions in $\bar\Delta$ in order for the space to be closed under shifts by $t$ for all $t\in \mathbb{R}$.

The next definition adapts the shift operator to continuous time by taking functions in $\bar \Delta$ as a generalization of bi-infinite sequences in $\Omega$.

\begin{defn}\label{psi}
Let 

\begin{center}
$
\begin{array}{lr}
 \psi: \mathbb{R} \times  \Delta  \rightarrow  \Delta \\
 \,\,\,\,\,\,\,\,(t, x(\cdot)) \mapsto x(\cdot + t)
\end{array}$
\end{center}
\end{defn}

Note that $\psi$ satisfies the flow property:
$$\psi(s+t, x(k)) = x(k + s + t) = x((k + t) + s) = \psi(s, x(k + t)) = \psi(s, \psi(t, x(k))) .$$

We now use the following to establish a metric on the set of functions $\Delta$.

\begin{defn}\label{bardeltafunction}
We take the function $f$ to be defined as follows:
\begin{center}
$\begin{array}{lc}
f: \Delta \times \Delta \times \mathbb{Z} \rightarrow \mathbb{R} \\ 
\,\,\,\,(x, y, i) \mapsto \frac{1}{h}\displaystyle\int_{ih}^{(i+1)h}{\delta(x,y,t)dt}
\end{array}$
\end{center}
where 
$$ \delta(x,y,t) =  \left \{ \begin{array}{cc}
1 & \mbox{if } x(t) \neq y(t)\\
0 & \mbox{if } x(t) = y(t)
\end{array} \right.$$
\end{defn}

\begin{thm}\label{bardeltafunctionmetric}
The function
\begin{center}
$
\begin{array}{lc}
 d: \Delta \times \Delta \rightarrow \mathbb{R} \\
 \,\,\,\,\,\,\,\,\,(x,y) \mapsto \displaystyle \sum_{i = - \infty}^{\infty}\frac{f(x,y,i)}{4^{|i|}}
\end{array} $
\end{center}
where $f(x,y,i)$ is defined as in Definition \ref{bardeltafunction}, is a metric on $\bar{\Delta}$.
\end{thm}

\begin{proof}
\ \\
\begin{enumerate}

\item (Non-negativity) $f(x, x, i) = 0$, for all $i \in \mathbb{Z}$.  Therefore, $d(x,x) = 0$.  For $x \neq y$, $f(x,y,i) \neq 0$ for at least one $i \in \mathbb{Z}$, and $f(x, y, i) \geq 0$ for all $i \in \mathbb{Z}$.  Therefore, $d(x,y) > 0$ for all $x \neq y$.
\item (Symmetry) Clearly, $\delta(x,y,t) = \delta(y,x,t)$ for all $t \in \mathbb{R}, x, y \in \Delta$.  So, $f(x,y,i) = f(y,x,i)$, $d(x,y) = d(y,x)$.
\item (Triangle inequality) Choose $x,y,z\in\Delta$.  If $x=z$, then as $d(x,z)=0$ and $d$ is nonnegative, then clearly for all $y\in\Delta$, $d(x,z)\leq d(x,y)+d(y,z)$.   

If $x\not =z$, then there exists $t\in\mathbb{R}$ such that $x(t)\not = z(t)$.  If $x(t) \neq z(t)$, then either $x(t) \neq y(t)$ or $y(t) \neq z(t)$.  Therefore, $\delta(x,z,t) = 1$ implies that $\delta(x,y,t) = 1$ and/or $\delta(y,z,t) = 1$, so $f(x,z,i) \leq f(x,y,i) + f(y,z,i)$.  Since $d$ is a linear combination of $f$'s, $d(x,z) \leq d(x,y) + d(y,z)$.
\end{enumerate}
\end{proof}
Because we would like to consider $\Delta$ as essentially the continuous version of $\Omega$, it helps to establish a relationship between them.  We show below that the spaces $\Omega$ and $\bar{\Delta}$ (the version that does not contain all real time shifts of the functions) are isometrically isomorphic, and are in fact topologically conjugate as well.

\begin{prop} \label{isoiso}
The mapping 
$\sigma: \Omega \rightarrow \bar\Delta$ where $x\mapsto x(t)$ where $x(i)=x_i$ for all $i\in\mathbb{Z}$ is an isometric isomorphism (bijection).
\end{prop}

\begin{proof}
By the construction of $\bar\Delta$, $\sigma$ is clearly bijective.\\
To show that $\sigma$ is an isometry, it suffices to show that $f(x,y,i) = \bar{f}(x_i,y_i)$, where
$$ \bar{f}(x_i,y_i) =   
\left\{ \begin{array}{cc}
1 & \mbox{if } x(i) \neq y(i)\\
0 & \mbox{if } x(i) = y(i)
\end{array} \right.$$
since the bi-infinite sums for $d$ and $\bar{d}$ are identical.  Note that
$$f(x,y,i) = \frac{1}{h}\int_{ih}^{(i+1)h}{dt} = 1 = \bar{f}(x_i, y_i)$$
for $x(i) \neq y(i)$.
$$ f(x,y,i) = \frac{1}{h}\int_{ih}^{(i+1)h}{0 \cdot dt} = 0 = \bar{f}(x_i, y_i)$$
for $x(i) = y(i)$.
So indeed, $f(x,y,i) = \bar{f}(x_i,y_i)$, and $d(x,y) = \bar{d}(\{x_i\},\{y_i\})$.

\end{proof}

Before we show the topological conjugacy of $\Omega$ and $\bar\Delta$, we first introduce the concept.
\begin{defn}\label{topconj}
A flow $\varphi$ on a space $X$ is said to be topologically conjugate to a flow $\psi$ on a space $Y$ if there is a homeomorphism $h:X\rightarrow Y$ such that $\psi(h(x),t)=h(\phi(x,t)).$
\end{defn}
Systems that are topologically conjugate can be shown to exhibit the same properties.  For instance, given the dynamics on $\Omega$, we then are already familiar with the dynamics on $\bar\Delta$.  Thus, these two systems are essentially equivalent. It is for this reason that we demonstrate this now.

\begin{thm}
$(\Omega,\Phi)$ and $(\bar{\Delta},\psi|_{\mathbb{Z}h})$ are topologically conjugate.
\end{thm}
\bpf
Let $\sigma:\Omega\rightarrow\bar\Delta$ be defined as in Proposition \ref{isoiso}.  Because $\sigma$ has been shown to be a bijective isometry, it is automatically a homeomorphism.  We thus claim that $\sigma$ is the homeomorphism that satisfies the definition of topologically conjugate given by Definition \ref{topconj}; that is, 
$$\psi(\sigma(x),nh)=\sigma(\Phi(x,nh))$$
for all $n\in\mathbb{Z}.$ Let $i\in\mathbb{Z}$; if we can show that $\sigma(\Phi(x,n))(ih)=\psi(\sigma(x),nh)_{ih}$, then the proof is complete.  Notice that the $i$th component of $\Phi(x,n)$ is $x_{i+n}$.  Thus, $\sigma(\Phi(x,n))(ih)=x_{i+n}$.  Similarly, $\sigma(x)((i+n))=x_{i+n}$, and thus $\psi(\sigma(x),nh)_{ih}.$.  Therefore, the result holds, and $(\Omega,\Phi)$ and $(\bar{\Delta},\psi|_{\mathbb{Z}h})$ are topologically conjugate.
\epf

Because of the topological conjugacy, $\bar\Delta$ inherits a lot of properties from $\Omega$, many of which we go through here.  Notice that many of the proofs are very similar, and many rely on $\Omega$ having the same properties.  However, we can not rely on topological conjugacy alone because the space that we are really interested in is $\Delta$, not $\bar\Delta$, which is a much bigger space containing $\Delta$ as a framework, but is not topologically conjugate to $\Omega$.  For instance, below we discuss the continuity of the shift.

\begin{lem}\label{psicontinuous}
$\psi_t$ is continuous for all $t \in \mathbb{R}$.
\end{lem}

\begin{proof}
Given $x, y \in \Delta$, we need to show that for all $\epsilon > 0$, there exists $\delta > 0$ such that
$$d(x,y) < \delta \Rightarrow d(\psi_t(x),\psi_t(y))<\epsilon.$$
Given any $\epsilon >0$, take $\delta = \frac{\epsilon}{4^n}$, where $n = \left\lceil|\frac{t}{h}|\right\rceil$, the least integer greater than the absolute value of $\frac{t}{h}$.  It is useful to rewrite $d(x,y)$ in the form
$$ d(x,y) = \frac{1}{h} \int_{-\infty}^{\infty}{\frac{1}{4^{\left |\lfloor \frac{t'}{h}\rfloor \right |}}\delta(x,y,t')dt'}$$
where $\delta$ is as defined above.  Given this, we can write
$$ d(\psi_t(x),\psi_t(y)) = \frac{1}{h} \int_{-\infty}^{\infty}{\frac{1}{4^{\left | \lfloor \frac{(t+t')}{h}  \rfloor \right |}}\delta(x,y,t')dt'}$$
And,
$$ \frac{1}{4^{\left | \lfloor \frac{(t+t')}{h} \rfloor \right |}} \leq 4^{\left \lceil |\frac{t}{h}| \right \rceil} \frac{1}{4^{\left | \lfloor \frac{t'}{h} \rfloor \right|}}$$

So, $$ d(\psi_t(x),\psi_t(y)) = \frac{1}{h} \int_{-\infty}^{\infty}{\frac{1}{4^{\left |\lfloor \frac{(t+t')}{h}  \rfloor \right |}}\delta(x,y,t')dt'}$$
$$ \leq 4^{\left \lceil |\frac{t}{h}| \right \rceil} \frac{1}{h} \int_{-\infty}^{\infty}{\frac{1}{4^{\left |\lfloor \frac{t'}{h} \rfloor \right|}}\delta(x,y,t')dt'}$$
$$ = 4^{\left \lceil |\frac{t}{h}| \right \rceil} d(x,y)$$
$$ < 4^{\left \lceil |\frac{t}{h}| \right \rceil} \delta$$
$$ = \epsilon$$
\end{proof}

Similarly to $\Omega$, we have compactness of $\Delta$. (We already know $\bar\Delta$ is compact since is it homeomorphic to $\Omega$.)

\begin{lem}\label{deltabarcompact}
 $\Delta$ is compact.
\end{lem}

\begin{proof}
We will show that given any sequence $\{x^n\}, n \in \mathbb{N}$ of functions $x^n \in \Delta$, there exists a subseqence converging to some $x \in \Delta$.  To do this, we consider the space $\Delta$ to be the product of a circle of length $h$ with the set of allowable bi-infinite sequences, $S^1 \times \Omega$, where $S^1 \equiv \mathbb{R}$ mod $h$; that is, $\Delta \sim S^1\times\overline{\Delta} \sim \Omega\times S^1$ (not with the product topology, however).  An element $x^n$ of $\Delta$ identifies with an element of $S^1 \times \Omega$ by taking $y^n \in \Omega$ to be the sequence of constant values of $x^n$, with $x^n(0) \equiv y^n_0, x^n(h) \equiv y^n_1,$ etc., and taking $\tau \in [0,h)$ to be the unique offset so that $x^n(t-\tau) \in \Delta$.

$S^1$ is compact.  Therefore, given the sequence $\{x^n\} \in S^1 \times \Omega$, there exists a subsequence $\{x^{n_k}\}, k \in \mathbb{N}$ for which the offsets $\{\tau^{n_k}\}$ converge to a value in $[0,h)$.  Therefore, there exists a convergent subsequence $\{\tau^{n_k}\}$ of $\{\tau^n\}$.

For this subsequence $\{x^{n_k}\}$, we want to show that there exists a subsequence $\{x^{n_{k_j}}\}, j \in \mathbb{N}$ such that the bi-infinite sequences $\{y^{n_{k_j}}\}$ converge.  We do this inductively, beginning with the subsequence $\{y^{n_{k_j}}_0\}$.  We know that  $\{y^{n_k}_0\}$ is an infinite sequence of finitely many values, since the state space $S$ is finite.  Therefore, by the pigeonhole principle, there is one value that is repeated infinitely many times.  Take $\{y^{n_{k_j}}_0\}$ to be this value, $s_0$, so that $x^{n_{k_j}}(0) = s_0$ for all $j \in \mathbb{N}$.

Now, we induct.  Given a subsequence of $\{x^{n_k}\}$ that converges at $t = 0, h, -h, 2h, -2h, ..., $ $mh, -mh$, we deduce that there must be a subsequence of this subsequence with one value $x^{n_k}((m+1)h) = s_{(m+1)h}$ repeated infinitely many times, and likewise for $x^{n_k}(-(m+1)h) = s_{-(m+1)h}$.  In this manner, we get an infinite subsequence $\{y^{n_{k_j}}\}$, hence $\{x^{n_{k_j}}\}$, converging to a function that is piecewise constant on $[\tau + nh, \tau + (n+1)h), n \in \mathbb{Z}, \tau \in [0,h]$, with values in $S$.

Finally, we have to show closure.  That is, we need to show that transitions $x(mh) \rightarrow x((m+1)h)$ in our limit function are allowable.  Otherwise, all we would have shown is compactness of $\bar{\Lambda}$, rather than compactness of $\Delta$.  Suppose that $x \notin \Delta$.  Then, there exists some $m \in \mathbb{Z}$ such that the transition $x(mh) \rightarrow x((m+1)h)$ is not allowed.  But, since $\{x^{n_{k_j}}\}$ converges to $x$, we can take $N$ large enough so that $j > N \Rightarrow x^{n_{k_j}}(mh) = x(mh), x^{n_{k_j}}((m+1)h) = x((m+1)h)$.  And, $x^{n_{k_j}} \in \Delta$, so the transition $x^{n_{k_j}}(mh) \rightarrow x^{n_{k_j}}((m+1)h)$ must be allowable.  This is a contradiction.  Therefore, $\{x^{n_{k_j}}\} \rightarrow x \in \Delta$.

\end{proof}

\subsection{Morse Sets and Topological Chaos}

For the following definitions and Proposition \ref{order}, taken from \cite{Kliemann}, let $X$ be a compact metric space with an associated flow $\Phi$.  These definitions have already been introduced in Chapter 3; we merely restate them here to remind the reader.

\begin{defn}\label{invariant}
A set $K \subseteq X$ is called invariant if $\Phi(t,x) \in K$ for all $x \in K, t \in \mathbb{R}$.
\end{defn}

\begin{defn}\label{isolated}
A set $K \subseteq X$ is called isolated if there exists a neighborhood $N$ of $K$ (i.e. a set $N$ with $K \subset$ int $N$) such that $\Phi(t,x) \in N$ for all $t \in \mathbb{R}$ implies $x \in K$.
\end{defn}

\begin{defn}\label{morsedecomp}
A Morse Decomposition on $X$ is a finite collection $\{\mathcal{M}_i, i = 1,...,n\}$ of non-void, pairwise disjoint, invariant, isolated, compact sets such that
\begin{enumerate}
\item For all $x \in X, \omega(x), \alpha(x) \subseteq \displaystyle \bigcup_{i=1}^{n}\mathcal{M}_i$.
\item If there exist $\mathcal{M}_0, \mathcal{M}_1, ..., \mathcal{M}_l$ and $x_1,...x_l \in X \setminus \displaystyle \bigcup_{i=1}^{n}\mathcal{M}_i$ with $\alpha(x_j) \subseteq \mathcal{M}_{j-1}$ and $\omega(x_j) \subseteq \mathcal{M}_j$ for $j = 1,...,l$, then $\mathcal{M}_0 \neq \mathcal{M}_l$.  This condition is equivalent to to the statement that there are no cycles between the sets of the Morse decomposition.
\end{enumerate}
The sets $\mathcal{M}_i$ above are called Morse sets.
\end{defn}

\begin{prop}\label{order}
The relation $\preceq$ given by
$$
\mathcal{M}_i \preceq \mathcal{M}_k \mbox{ if there are } \mathcal{M}_i,\mathcal{M}_{j_1},...,\mathcal{M}_{j_l} =
 \mathcal{M}_k \mbox{ and } x_1,...,x_l \in X $$
 $$\mbox{ with } \alpha(x_m) \subseteq \mathcal{M}_{j_{m-1}} \mbox{ and } \omega(x_m) \subseteq \mathcal{M}_{j_m} \mbox{ for } m = 1,...,l.
$$
is an order (satisfying reflexivity, transitivity, and antisymmetry) on the Morse sets $\mathcal{M}_j$ of a Morse decomposition.
\end{prop}

The proof of this proposition can be found in \cite{Kliemann}.

Also, similar to the concept of lifts of strongly connected components of $G$ in $\Omega$, we have lifts of strongly connected components in $\Delta$.

\begin{defn}\label{lifts}
The lift $\Delta_C \subseteq\Delta$ of a strongly connected component $C$ of $G$ is defined by
$$
\Delta_C \equiv \{ f \in \bar\Delta | f(t) \in C \, \mbox{ for all }\, t \in \mathbb{R} \}
$$
$\bar{\Delta_C}$ is defined as 
$$
\bar\Delta_C \equiv \Delta_C \cap \bar\Delta.
$$
That is to say, $\Delta_C$ contains all real time shifts of functions in $\overline{\Delta_C}$. 
\end{defn}
In other words, given a strongly connected component $C$, $\Delta_C$ is the set of all functions $f\in\Delta$ whose ranges are contained in $C$.

Unsurprisingly, these lifts display many of the same qualities as those exhibited in $\Omega$, the first being that they form a Morse Decomposition for $(\Delta,\psi).$ 

\begin{thm}\label{morse1}
The lifts of the strongly connected components $\Delta_C$ are Morse sets for the dynamical system $\psi$.
\end{thm}

\begin{proof}
We check each of the conditions in turn.
\begin{enumerate}
\item \textbf{Non-void}
Since the empty set is not a strongly connected component, the lift of any strongly connected component must be non-empty.
\item \textbf{Pairwise disjoint}
Suppose that there exists $f \in \Delta_C, \Delta_{C'}$ with $C \neq C'$.  Then, $f(0) \in C, C'$.  But, by the maximality of strongly connected components, $f(0) \in C, C'$ implies $C = C'$.  So, $\Delta_{C} = \Delta_{C'}$.
\item \textbf{Invariant}
By construction of $\Delta$, $\psi(t,f) \in \Delta$ for all $t \in \mathbb{R}, f \in \Delta$.  And, if $f(s) \in C$ for all $s \in \mathbb{R}$, then $\psi(t,f)(s) = f(t+s) \in C$.  So, $\psi(t,f) \in \Delta_C$ for all $t \in \mathbb{R}, f \in \Delta_C$.
\item \textbf{Isolated}
Pick $\epsilon = 1/4$.  Suppose that there exists $g \notin \Delta_C$ such that for some $f \in \Delta_C$, $d(g,f) < \epsilon$.  Since $g \notin \Delta_C$, there exists $t_0$ such that $g(t_0) \notin C$.  Let $g' = \psi(-t_0,g)$, so that $g'(0) \notin C$.  But then, $g'$ differs from any function in $\Delta_C$ on at least some interval of length $h$ containing $0$.  The distance, therefore, between $g'$ and any function in $\Delta_C$ must be greater than $1/4$.  So, given any $g \notin \Delta_C$ but within $1/4$ of $\Delta_C$, there exists $t_0$ such that $d(\psi(t_0, g),f') > \frac{1}{4}$ for any $f' \in \Delta_C$.  Hence, $\Delta_C$ is isolated.  
\item \textbf{Compact}
By an argument similar to that for compactness of $\Delta$ and by compactness of $\Omega_C$, $\Delta_C$ is compact.

\item \textbf{No cycles}
Again, this is similar to the corresponding proof in \cite{discretesystems}.  Suppose that there exist $f, g \in \bar{\Delta}$ such that $\alpha(f) \subseteq \bar{\Delta}_C$, $\alpha(g) \subseteq \bar{\Delta}_{C'}$ and $\omega(g) \subseteq \bar{\Delta}_C$, $\omega(f) \subseteq \bar{\Delta}_{C'}$.  Then, since all the transitions in $f,g$ must be allowable, there must exist an admissible path from $C$ to $C'$ as well as one from $C'$ to $C$.  But, this contradicts maximality of strongly connected components.  So, no such cycle exists.
\end{enumerate}

\end{proof}


It turns out that the flow within each of the lifts of strongly connected components is topologically transitive as well, just like as in $\Omega$.  We first restate the definition of topological transitivity below.

\begin{defn}\label{topologicallytransitive}
A flow on a metric space $X$ is called topologically transitive if there exists $x \in X$ such that $\omega(x) = X$.
\end{defn}

\begin{lem}\label{omegalimit}
Given any strongly connected component $C$, there exists $f^* \in \Delta_C$ such that $\omega(f^*) = \Delta_C$ (i.e. $\psi$ is topologically transitive on lifts of strongly connected components).
\end{lem}

\begin{proof}
It has been shown in \cite{discretesystems} that for the discrete system, there exists $x^* \in \Omega_C$ such that $\omega(x^*) = \Omega_C$.  This proof is also state in Chapter 2. By the correspondence $\sigma$ as defined in Proposition \ref{isoiso} between sequences in $\Omega$ and functions in $\Delta$, there exists $f^* \in \Delta_C$ given by $$ f^*(nh) = x^*_n, n \in \mathbb{Z}$$ such that $\Delta_C \subseteq \omega(f^*)$.  And, since $\bar{\Delta}_C$ is given by the shifts $\psi(t,\Delta_C)$, it is clear that $\Delta_C \subseteq \omega(f^*)$.  $\Delta_C$ is invariant by Theorem \ref{morse1}, so $\omega(f^*) \subseteq \Delta_C$, and $\omega(f^*) = \Delta_C$.
\end{proof}

Since the $\omega$-limit sets of a point on a compact space are connected, we get the following corollary.

\begin{cor}\label{connected}
$\Delta_C$ is connected.
\end{cor}

Once again, just as in $\Omega_C$, we see that these points are dense in $\Delta_C$ for any strongly connected component $C$.

\begin{prop}\label{dense}
The set of all functions $f^*$ satisfying $\omega(f^*) = \Delta_C$ is dense in $\Delta_C$
\end{prop}

\begin{proof}
Given $f \in \Delta_C$, there exists $f^*$ such that $f \in \omega(f^*)$, by Lemma \ref{omegalimit}.  Therefore, given $\epsilon > 0$, there exists $t \in \mathbb{R}$ such that $d(\psi(t,f^*),f) < \epsilon$.   $\omega(f^*) = \Delta_C$ implies $\omega(\psi(t,f^*)) = \Delta_C$. So, for any $f \in \Delta_C$ and any $\epsilon > 0$, there exists a function $\psi(t,f^*) \in \Delta_C$ with $d(\psi(t,f^*),f) < \epsilon$ and $\omega(\psi(t,f^*)) = \Delta_C$.
\end{proof}

The above propositions serve to show that the lifts of strongly connected components $\Delta_C$ can not be broken up into smaller invariant components, which serves useful in showing that this Morse Decomposition given by the lifts of strongly connected components is in fact the finest Morse Decomposition that exists on $\Delta,\psi$. 

\begin{defn}\label{finer}
A Morse Decomposition $\{\mathcal{M}_1,...,\mathcal{M}_n\}$ is called finer than a Morse Decomposition $\{\mathcal{M}'_1,...,\mathcal{M}'_l\}$ if for all $j \in \{1,...,l\}$ there exists $i \in \{1,...,n\}$ such that $\mathcal{M}_i \subseteq \mathcal{M}'_j$, where containment is proper for at least one $j$.
\end{defn}

\begin{thm}\label{finest}
The lifts of the strongly connected components $\Delta_C$ form a finest Morse decomposition on $\Delta$.
\end{thm}

\begin{proof}
Suppose there exists a finer Morse decomposition, $\mathcal{M_1},\ldots,\mathcal{M}_k$.  Then, for some strongly connected component $C$, there exists a Morse set $\mathcal{M}_i \subsetneq \Delta_C$, a proper containment.  By the definition of a Morse set, $\mathcal{M}_i$ must contain the $\omega$-limit sets of $\Delta_C$.  However, by Lemma \ref{omegalimit}, there exists $f^* \in \Delta_C$ such that $\omega(f^*) = \Delta_C$.  Therefore, $\Delta_C \subseteq \mathcal{M}_i$, so $\mathcal{M}_i$ is not a proper subset of $\Delta_C$.  Thus, no finer Morse decomposition exists.

\end{proof}
Lastly, as one would expect, the behavior of the flow on the lifts of strongly connected components exhibits chaotic behavior and sensitive dependence on initial conditions as well, just as in $\Omega$.

\begin{defn}\label{sensitivedependence}
A flow $\Phi$ on a metric space $X$ has sensitive dependence on initial conditions if there exists $\delta > 0$ such that for every $x \in X$ and every neighborhood $B$ of $x$, there exists $y \in B$ and $t > 0$ such that $d(\Phi_t(x),\Phi_t(y)) > \delta$.
\end{defn}

\begin{defn}\label{chaotic}
A flow on a metric space $X$ is chaotic if it has sensitive dependence on initial conditions, density of periodic points, and is topologically transitive.
\end{defn}

\begin{lem}\label{sensdep}
Consider a graph $G$ consisting of a single strongly connected component $C$ for which the out-degree of at least one vertex is at least two.  Then, $\psi$ on $\Delta$ has sensitive dependence on initial conditions.

\end{lem}

\begin{proof}
Take $\delta = \frac{1}{2}$.  Given $x \in \Delta$, we construct a function $y \in \Delta$ such that $x$ and $y$ are discontinuous at the same times mod $h$.  Given $\varepsilon > 0$, take $N$ large enough so that
$$\displaystyle \sum_{i=-\infty}^{-N}{\frac{1}{4^{|i|}}} + \displaystyle \sum_{i=N}^{\infty}{\frac{1}{4^{|i|}}} < \varepsilon.$$
Thus, taking $x(t) = y(t)$ on $t \in [-Nh,Nh]$ ensures that $d(x,y) < \varepsilon$.  Now, we just need to show that there exists $m>N \in \mathbb{R}$ and $y\in \Delta$ so that $y(t) \neq x(t)$ for all $t \in [mh,(m+1)h)$.  This would imply that $d(\psi(mh,x),\psi(mh,y)) \geq 1 > \delta$.
To show that such an $m$ exists, let $\gamma_1$ denote the vertex with out-degree greater than one.  If there does not exist a $t > Nh$ such that $x(t) = \gamma_1$, then given $x(Nh) = \gamma_N$, let $y(t), t > Nh$ follow a path from $\gamma_N$ to $\gamma_1$.  Such a path must exist since $G$ consists of a single strongly connected component, so there exists a path between any two vertices in $G$.  Thus, we would have $x(t_0+t) \neq \gamma_1, y(t_0+t) = \gamma_1$ for some $t_0 > Nh, t \in [0,h)$, so we can take $m = \frac{t_0}{h}$. If there does exist a $t > Nh$ such that $x(t) = \gamma_1$, then define $\gamma_2 = x(t+h)$, and take $t_1$ so that $x(t_1+t)=\gamma_1$ for $t \in [0,h)$.  Since the out-degree of $\gamma_1$ is greater than one, there exists an edge from $\gamma_1$ to some other vertex, $\gamma_3$ (note that it is possible that either $\gamma_1 = \gamma_2$ or $\gamma_1 = \gamma_3$, but not that $\gamma_2 = \gamma_3$).  Set $y(t_1+h) = \gamma_3$, and take $m = \frac{t_1+h}{h}$.

\end{proof}

Note that Lemma \ref{sensdep} can also be proven by noticing that density of periodic points and topological transitivity imply sensitivity of initial conditions, and thus chaos. Considering together Lemmas \ref{omegalimit} and \ref{sensdep} yields the following result.
\begin{thm}\label{chaos}
Consider a graph consisting of a single strongly connected component for which the out-degree of at least one vertex is greater than one.  Then, $\psi$ is chaotic on $\Delta$.
\end{thm}
\begin{cor}
$\psi$ is chaotic on lifts of strongly connected components (where the out-degree of at least one vertex is greater than one).  If the out degree of every vertex is exactly one, then the lift of the strongly connected component is a single periodic orbit. 
\end{cor}
Finally, with confirmation that behavior in $\Delta$ mirrors that of $\Omega$ as explained in Chapter 3, we move on to considering the hybrid system defined on $M\times\Delta$.

\section{Deterministic Hybrid Systems}
Now that the behavior on $\Delta$ has been determined given a natural number $n$ and an $N$-graph on $n$ vertices, we consider the action of a function $f\in\Delta$ on a set of $n$ dynamical systems.

Consider an $N$-graph $G$ with $n$ vertices.  Take a collection of $n$ dynamical systems $\{\phi_1,...,\phi_n\}$ on a compact space $M \subset \mathbb{R}^d$, where each vertex of $G$ corresponds to one dynamical system $\phi_i$.  Take $f \in \Delta$.  Define $\varphi(t,x,f):\mathbb{R} \times M \times \Delta \rightarrow M$ by
$$
\varphi (t,x,f) = \varphi_t(x,f)= \phi_{i_m}(\tau,(\phi_{i_{m-1}}(h,...\phi_{i_{1}}(h,x))) 
$$
where $\tau\equiv t \mod h$, and $f(s)=i_k$ for $s\in [(k-1)h,kh)$. 

Note also that use of $\phi_i^{-1}$ defines this system backwards in time as well, and that $\phi$, as a composition of continuous functions, remains continuous with respect to $x$. 

Thus, $\varphi(t,x,f)$ is given by the flow along the dynamical system $\phi_i$ during the period of time for which $f = i$.  With this, we can explicitly define our deterministic hybrid system.  Consider 

\begin{equation}\label{dynsys}
\Phi_t(x_0,f_0) \equiv \left (
\begin{array}{cc}
\varphi_t(x_0,f_0) \\
\psi_t(f_0)
\end{array} \right ) 
\end{equation}

with initial conditions $f_0 \in \Delta$, and $x_0\in M$.  Let $\psi_t(f_0) = f_t$ and notice that

\begin{center}
$\Phi_0(x_0,f_0) = \left (
\begin{array}{cc}
x_0 \\
f_0
\end{array} \right ) $
\end{center}

Then,
$$\Phi_{t+s}(x_0,f_0) 
= 
\left (
\begin{array}{cc}
\varphi_{t+s} (x_0, f_0) \\
f_{t+s}
\end{array} \right ) 
 = 
\left (
\begin{array}{cc}
\varphi_t (\varphi_s(x_0, f_0), f_s) \\
f_t \circ f_s
\end{array} \right ) 
= 
\Phi_t \circ \Phi_s(x_0,f_0) .$$
Thus, $\Phi_t$ is in fact a flow, so the deterministic hybrid system is a dynamical system.\\
\indent This explanation is rather unintuitive and bulky.  It is easier to consider a less rigorous definition of the system.  Consider a point $(x,f)\in M\times\Delta.$  We first consider the dynamical system that is given to us by $f(0)$.  As time moves forward, the orbit of $x$ is given by that dictated by the dynamical system corresponding to $\psi(f,t)(0)$.  Recall that $\psi$ simply shifts the function $f$ to the right.  $f$ is piecewise constant; at time $h$, we may have that for all $\varepsilon>0,$ $\psi(f,h)(0)\not = \psi(f,h-\varepsilon)(0).$  If the function changes values, we then instantaneous switch which dynamical system is dictating the orbit of $x$.  This could then result in a non-smooth orbit.  This continues, with a possible change in dynamical system on $M$ occurring after a time interval of length $h$. This can be further understood via the following commuting diagram (where $\pi_1$ and $\pi_2$ are the usual projection mappings):
\begin{center}
\begin{tikzcd}
\mathbb{R}\times M\times\Delta \arrow[r,"\Phi"] \arrow[d,"\pi_1"]
	& M\times\Delta \arrow[d,"\pi_2"]\\
\mathbb{R}\times\Delta \arrow[r,"\varphi"]
	&\Delta
\end{tikzcd}
\end{center}

\indent For the set $M\times\Delta$, we use the metric induced by the \textbf{$L_1$} norm; that is, the metric in $M\times \Delta$ is given by the sum of the metrics used in $M$ (usually given by the standard Euclidean norm in $\mathbb{R}^n$, and the metric used in $\Delta$.)  Note then that by Tychonoff's  Theorem that $M\times \Delta$ is compact, and that $\Phi$ is continuous (as it is continuous in each of its components separately).  \\
\indent It is helpful to have a terminology that explains the behavior in $M\times\Delta$, in that what happens on $\Delta$ is independent, but that the flow on $M$ is dependent on what happens in $\Delta$.  

\begin{defn} A flow $\pi$ on a product space $X\times Y$ is said to be a skew-product flow if there exist continuous mappings $\phi:X\times Y\times T\rightarrow X$ and $\sigma:(Y\times T\rightarrow Y)$ such that 
$$\pi(x,y,t)=(\psi(x,y,t),\sigma(y,t))$$
where $\sigma$ itself is a flow on $Y$ and $T$ is a group that represents time values (for our purposes, we shall take $T=\mathbb{R}$.) \cite{skewproduct}
\end{defn} 

It should not be too hard to see that our system is a skew-product flow, where $X$ as in the definition above corresponds to our space $M$, and $Y$ in the definition above corresponds to $\Delta$.  \\
\indent Now that we have an understanding of the behavior of our system $(M\times\Delta, \Phi)$, we turn to examining some recurrence concepts.  We begin by adapting our understanding of $(\varepsilon,T)$-chains to fit this new situation.

\begin{defn} \label{chain set} A set $E\subset M$ is called a chain set of a system if (i) for all $x\in E$ there exists $f\in\Delta$ such that $\varphi(t,x,f)\in E$ for all $t\in\mathbb{R}$, (ii) for all $x,y\in E$ and for all $\varepsilon,T>0$ there exist $n\in\mathbb{N}$, $x_0,\ldots,x_n\in M$, $f_0,\ldots,f_{n-1}\in\Delta$ and $t_0,\ldots,t_{n-1}\geq T$ with $x_0=x$, $x_n=y$, and 
$$d(\varphi(t_j,x_j,f_j),x_{j+1})\leq\varepsilon\mbox{ for all }j=0,\ldots,n-1.$$
Such a sequence is called an $(\varepsilon,T)$-chain from $x$ to $y$. 
\end{defn}

If there exists a $(\varepsilon,T)$-chain from $x$ to $y$ and from $y$ to $x$, we say that $x$ and $y$ are chain equivalent.\\
\indent It is important to note that chain sets as defined are distinct from the usual chain recurrent components seen in dynamical systems.  This is because the behavior on $M$ is not, in and of itself, a dynamical system, and thus the concept of chain recurrence is not applicable here.  However, in Lemmas \ref{chaindisjoint}, \ref{compact},  and \ref{chainconnected}, we do demonstrate that chain sets, when considered as maximal components, exhibit nice properties that we would expect of recurrent sets.

\begin{lem}\label{chaindisjoint}
Chain sets are pairwise disjoint.
\end{lem}
\begin{proof}
Let $E_1$ and $E_2$ be two chain sets, and let $x\in E_1\cap E_2$ (that is, $E_1$ and $E_2$ are not disjoint).  Then let $y\in E_1$ and $z\in E_2$.  Then, given $\varepsilon, T>0$, by definition of a chain set there exist $n\in\mathbb{N}$, $x_0,\ldots,x_n\in M$, $f_0,\ldots,f_{n-1}\in\Delta$ and $t_0,\ldots,t_{n-1}\geq T$ with $x_0=y$, $x_n=x$, and 
$$d(\varphi(t_j,x_j,f_j),x_{j+1})\leq\varepsilon\mbox{ for all }j=0,\ldots,n-1.$$  Similarly, for all $\varepsilon, T>0$ there exist $n\in\mathbb{N}$, $x_0,\ldots,x_n\in M$, $f_0,\ldots,f_{n-1}\in\Delta$ and $t_0,\ldots,t_{n-1}\geq T$ with $x_0=x$, $x_n=z$, and 
$$d(\varphi(t_j,x_j,f_j),x_{j+1})\leq\varepsilon\mbox{ for all }j=0,\ldots,n-1.$$
Thus, the concatenation of these two $\varepsilon, T$-chains results in a $\varepsilon, T$-chain from $y$ to $z$, and thus $y$ and $z$ are in the same chain set, and thus, since the choice of $y$ and $z$ was arbitrary, $E_1=E_2$.  
\end{proof}

\begin{lem}  \label{compact}
Chain sets are compact.
\end{lem}
\begin{proof}
Let $E$ be a chain set, and let $x\in M$ be a limit point of $E$.  Then there exists a sequence in $E$, $\{x_i\}_{i=1}^\infty$ such that $x_i\rightarrow x$.  Let $y\in E$, and let $\varepsilon, T>0$ be given.  Then there exists $N\in\mathbb{N}$ such that $d(x_N,x)<\varepsilon$.  By definition of a chain set, there exist $n\in\mathbb{N}$, $x_0,\ldots,x_n\in M$, $f_0,\ldots,f_{n}\in\Delta$ and $t_0,\ldots,t_{n}\geq T$ with $x_0=y$, $\varphi(t_n,x_n,f_n)=x_N$, and 
$$d(\varphi(t_j,x_j,f_j),x_{j+1})\leq\varepsilon\mbox{ for all }j=0,\ldots,n-1.$$  
Thus, since $d(\varphi(t_n,x_n,f_n),x)<\varepsilon$, setting $x_{n+1}=x$ gives an $\varepsilon, T$-chain from $y$ to $x$, and similarly from $x$ to $y$, and thus $x\in E$, and thus $E$ is closed.  Since $E$ is then a closed subset of the compact set $M$, $E$ is thus compact. 
\end{proof}

\begin{lem}\label{chainconnected}
Chain sets are connected.
\end{lem}
\begin{proof}
Let $E$ be a chain set and let $A,B$ be open sets such that $E\subset A\cup B$ and $A\cap B=\emptyset$.  If $\inf\{d(a,b)|a\in A, b\in B\}>0$, then there exists $\varepsilon<\inf\{d(a,b)|a\in A, b\in B\}$ and thus there exists no $\varepsilon, T-$chain from any $a\in A$ to any $b\in B$, and thus one of $A$ or $B$ must be empty.  If $\inf\{d(a,b)|a\in A, b\in B\}=0$, then there exists some $x$ such that $\inf\{d(a,x)|a\in A\}=0$ and $\inf\{d(b,x)|b\in B\}=0$.  Since by Lemma \ref{compact}, $E$ is closed, this implies that $x\in E$, and thus $x\in A$ or $x\in B$.  Without loss of generality, let $x\in A$.  Since $\inf\{d(b,x)|b\in B\}=0$, this implies that if $B$ is nonempty, for every neighborhood of $N$ of $X$, $N\cap B\not= \emptyset$.  However, since $A$ is open, there is a neighborhood $N$ of $x$ such that $N\subset A$.  This then implies that $A\cap B\not = \emptyset$, a contradiction.  Thus $B$ is empty, and $E$ is connected.

\end{proof}

Note that the above lemma does not show that chain sets are necessarily path connected; indeed, the following is an example of a chain set which is not path connected. 

\begin{ex}
Let $G$, the graph governing $\Delta$, be the complete graph on 2 vertices, and let $M=\{(x,f(x))| x\in(0,1/2\pi], f(x)=\sin(1/x)\}\cup\{0\}\times[-1,1]$, otherwise known as the Topologist's Sine Curve.   Let the two systems defined on $M$ be given by the following differential equations:
\begin{eqnarray}
A: \dot{x}&=&-x(1/2\pi-x)\\
B: \dot{x}&=&x(1/2\pi-x)
\end{eqnarray}
(Note that it is sufficient to describe the dynamics of the systems with just the behavior in the $x$-coordinate alone as, except where $x=0$ - which consists entirely of fixed points - there is exactly one $y$-value for each $x$-value.) Thus, in both systems, the set $\{0\}\times[-1,1]$ is entirely made up of fixed points, and in System A, along the set $\{(x,f(x))| x\in(0,1/2\pi], f(x)=\sin(1/x)\}$ the system moves from right to left (with a fixed point at $x=1/2pi$), the speed converging to zero as x approaches zero.  Similarly, in system $B$, along $\{(x,f(x))| x\in(0,1/2\pi], f(x)=\sin(1/x)\}$, the system moves left to right, with a fixed point at $x=1/2\pi$.  We then claim that the entirety of $M$ is forms one chain set.  
Let $x,y \in M.$ It should be clear that, if $x,y \in \{0\} \times [-1,1],$ or if $x,y \in \{(x,f(x))|x \in (0,1/2\pi],f(x)= \sin(1/x)\},$ then for all $\varepsilon,T > 0 $, there exists an $\varepsilon,T-chain$ from $x$ to $y.$ If $x \in \{0\}\times [-1,1]$ and $y \in \{(x,f(x))|x \in (0,1/2\pi],f(x) = \sin(1/x)\}$, then a chain can be formed by staying at x for a time of at least T, and then jumping by $\varepsilon$ onto a point $z \in \{(x,f(x))|x \in(0, 1/2\pi], f (x) = \sin(1/x)\}$ (since M is the closure of $\{(x, f (x))|x \in (0, 1/2\pi], f (x) = \sin(1/x)\}$, x is a limit point of  $\{(x,f(x))|x \in(0,1/2\pi],f(x) = \sin(1/x)\}$, and thus there exists a point z within $\varepsilon$ of x). Thus, since we have established that there is a chain from z to y, there is a chain from x to y. Similarly, if $x\in \{(x, f (x))|x \in (0, 1/2\pi], f (x) = \sin(1/x)\}$ and $y \in \{0\} \times [-1, 1]$, then, if we let z be a point on $\{(x, f (x))|x \in (0, 1/2\pi], f (x) = \sin(1/x)\}$ within $\varepsilon$ of y, then there exists a chain from x to z, and then jumping to y gives a chain from x to y. Thus, M is a chain set.

It is well known that M, the Topologist's Sine Curve, is a set that is connected but not path connected. A proof can be found in \cite{counterexamples}, pages 137-138.
\end{ex}
\noindent Thus, while chain sets are always connected, they may not exhibit path connectivity.  \\

\indent It is important to remember that chain sets are not the usual chain transitive sets.
This is because we can not consider the behavior on $M$ alone as a flow; it is dependent on orbits in $\Delta$. The following example shows why these concepts are not the same.

\begin{ex}
Consider the system where $\Delta$ is given by the complete graph on two vertices, and $M=[0,2]$.  Let the system corresponding to vertex $A$ be given by the differential equation:
$$\dot{x}=-x(x-1)(x-2),$$
and the system corresponding to vertex $B$ is given by 
$$\dot{x}=-x(x-2).$$
Both of these systems are bounded on either end by fixed points at 0 and 2.  System $A$ has a repelling fixed point at $1$, while in system $B$ on the interval $(0,2)$ the flow moves in the positive direction.  We claim that in this system, the interval $[0,1]$ is a chain set for all $T, \varepsilon>0$.  (The fixed point at $x=2$ is also rather trivially a chain set.) Note that on the open interval $(0,1)$, the flow moves in two different directions in each system.  Note further that the lift of $[0,1]$ is not equal to $\Delta\times [0,1]$, as $[0,1]$ is not invariant.  Then let $y,z\in [0,1)$.  Given $\varepsilon,T>0$, we wish to construct an $(\varepsilon,T)-$chain from $y$ to $z$.  Consider the sequence $a_n=\varphi(-nT,z,B).$ Notice then that 
$$\lim_{n\rightarrow\infty}a_n=0.$$
Let $N$ be such that $|0-a_N|<\varepsilon/2$. Notice that there exists a time $T'>T$ such that $$|0-\varphi(T',y,A)|<\varepsilon/2.$$ Thus, by the triangle inequality, 
$$|\varphi(T',y,A)-a_N|<\varepsilon.$$
Therefor, the sequence $y, \varphi(T',y,A),a_N,z$ forms an $(\varepsilon,T)-$chain from $y$ to $z$.  Since we know chain sets are closed by Lemma \ref{compact}, we know now that $[0,1]$ is a chain set for all $T,\varepsilon>0$. \\
\indent Now, in the above system, let the graph $G$ be the cycle on two vertices.  Then we claim that $[0,1]$ is no longer a chain set for all $T>0$.  Note now that the only functions in $\Delta$ are shifts of the periodic function that switches between $A$ and $B$ on intervals of length $h$.  Without loss of generality, let $z<y$ such that $\varphi(h,\varphi(h,z,A),B)\not=z$.  Such a point exists because for all $\varepsilon>0$, there exists a point $x\in(1/2,1)$ such that $\varphi'(t,x,A)<\varepsilon$ as the flow in system $A$ converges to 1.  However, since $\phi'(t,1,B)\not=0$, the flow is not symmetric, and thus we can not have that $\varphi(h,\varphi(h,z,A),B)=z$ for all $z\in(0,)$, and thus such a point $z$ exists.  We would like to show that there exist $\varepsilon,T$ such that there is no longer an $(\varepsilon,T)$-chain from $z$ to $y$.  Let us start at $z$ with system $A$.  Let $T=2h$, and $x_{2h}=\varphi(h,\varphi(h,z,A),B)$.  If $x_{2h}<z$, then pick $\varepsilon<|z-x_{2h}|$ - notice now that all solutions must be less than $z$ for all times $t>h$.  Similarly, if $x_{2h}>z$, we can choose $\varepsilon$ such that we can not reach any points less than $z$ with a particular $\varepsilon$. If $z=x_{2h}$, let $\varepsilon$ be small enough such that $\varphi(-h,1+\varepsilon,B)<z$. \\
\indent However, note that if we take $T=h$, since we are allowed to switch functions after each $\varepsilon$-jump, we are essentially in the same case as when the graph $G$ is complete (since we may jump by $\varepsilon=0$ and let $f(0)$ take either value $A$ or $B$ after the jump). Thus, in this case $[0,1]$ is an $(\varepsilon,h)$-chain set.  Note that this example then implies that, $(\varepsilon,h)$-chain sets and $(\varepsilon,T)$-chain sets for a general $T$ may not be equivalent, and thus Theorem \ref{h flow} does not apply in this case.
\end{ex}

The above example illustrates that $\varepsilon,h$-chain sets are like having the complete graph (see Lemma \ref{complete graph}); that is, if $\mathcal{E}\subset \Delta\times M$ is a maximal invariant chain transitive set, then $\mathcal{E}=\ell(\pi_M\mathcal{E})$, and $\pi_M\mathcal{E}$ is a chain set for all $T$ (see Theorem \ref{max chain set}).\\
\indent Thus, we see that the relationship between chain sets in $M$ and chain transitive sets in $M\times \Delta$ is rather complicated. As of right now, there is no general theory about the relationship between the two concepts.  Above we have explored certain examples of relationships, but future work may entail coming up with a more general result that relates the two.  In addition, concepts such as Poincar\'e recurrence and nonwandering sets could be explored within this context.  

\indent Since a chain set $E$ is a subset of $M$, it helps to have an extension of it to a set contained in $M\times \Delta$. 

\begin{defn}
Given $E\subset M$, the lift of $E$ to $M\times\Delta$ is given by
$$\ell(E)=\{(x,f)\in M\times\Delta,\Phi(t,x,f)\in E \text{ for all }t\in\mathbb{R}\}.$$
\end{defn}

\begin{defn}
A set $A\subset M$ is said to be invariant if for all $x\in A$, $\varphi(t,x,f)\in A$ for all $t\in\mathbb{R}$ and $f\in\Delta$.
A set $A\subset M$ is said to be forward invariant if for all $x\in A$, $\varphi(t,x,f,)\in A$ for all $t\in \mathcal{R}^+$ and $f\in\Delta$.  Similarly, a set $A\subset M$ is said to be backward invariant if for all $x\in A$, $\varphi(t,x,f,)\in A$ for all $t\in \mathcal{R}^-$ and $f\in\Delta$.
\end{defn}

\begin{remark}
Notice that if $E$ is invariant, $\ell(E)=\Delta\times E$. 
\end{remark}

Since $(M\times\Delta,\Phi)$ is a dynamical system on a compact set, it has chain transitive sets.  We would like to make connections between a chain transitive set contained in $M\times \Delta$ and a chain set that is a subset of $M$.  The following results relate the ideas of chain transitive sets, chain sets, projections, and lifts, and what properties are retained when projecting onto $M$ or lifting to $M\times \Delta$.

\begin{thm}\label{max chain set}
Let $\mathcal{E}\subset M\times\Delta$ be a maximal invariant chain transitive set for the flow.  Then $\pi_M\mathcal{E}$ is a chain set.
\end{thm}
\begin{proof}
Let $\mathcal{E}$ be an invariant, chain transitive set in $M\times\Delta$.  For $x\in\pi_M\mathcal{E}$ there exists $f\in\Delta$ such that $\varphi(t,x,f)\in\mathcal{E}$ for all $t$ by invariance.  Now let $x,y\in \pi_M\mathcal{E}$ and choose $\varepsilon, T>0$.  Then by chain transitivity of $\mathcal{E}$, we can choose $x_j, f_j, t_j$ such that the corresponding trajectories satisfy the required condition.  The proof is concluded by noticing that $\pi_M\mathcal{E}$ is maximal if $\mathcal{E}$ is maximal.

\end{proof}

\begin{lem}
Given a maximal invariant chain transitive set $\mathcal{E}\subset M\times\Delta$, $\mathcal{E}\subset\ell(\pi_M\mathcal{E})$.
\end{lem}
\begin{proof}
Let $(x,f)\in\mathcal{E}$.  Then $x\in \pi_M\mathcal{E}$.  Since $\mathcal{E}$ is invariant, $\Phi(t,x,f)\in \mathcal{E}$ for all $t\in\mathbb{R}$.  Thus, $\pi_M\Phi(t,x,f)=\varphi(t,x,f)\in\pi_M\mathcal{E}$ for all $t\in\mathbb{R}$.  This then implies that $(x,f)\in\ell(\pi_M\mathcal{E})$, by definition of the lift. 
\end{proof}

We then wondered if it was possible to establish a more general theory about chain sets and chain transitive sets, and how they are related via lifts and projections.  In order to accomplish this task, we made use of the following theorem, taken from \cite{Alongi}, Theorem 2.7.18.

\begin{thm}\label{h flow}
If $\phi^t$ is a flow on a compact metric space $(X,d)$ and $x,y\in X$, then the following statements are equivalent.
\begin{enumerate}
\item The points x and y are chain equivalent with respect to $\phi^t$.
\item For every $\varepsilon>0$ and $T>0$ there exists an $(\varepsilon,1)$-chain
$$(x_0,\ldots,x_n;t_0,\ldots,t_{n-1})$$
from x to y such that 
$$t_0+\cdots+t_{n-1}\geq T,$$
and there exists an $(\varepsilon,1)$-chain 
$$(y_0,\ldots,y_m;s_0,\ldots,s_{m-1})$$
from y to x such that
$$s_0+\cdots+s_{m-1}\geq T.$$
\item For every $\varepsilon>0$ there exists an $(\varepsilon,1)$-chain from x to y and a $(\varepsilon,1)$-chain from y to x.
\item The points x and y are chain equivalent with respect to $\phi^1$.
\end{enumerate}
\end{thm}

Notice then, that by this theorem, for chain sets $E$ such that $E=\pi_M(\mathcal{E})$, where $\mathcal{E}$ is the lift of $E$, it is sufficient to take Definition \ref{chain set} consider all chains where all $t_i$'s take the value $h$ (since the number 1 in part 4 of the above theorem is rather arbitrary).  However, this may not be true for all chain sets in general, as there exist chain sets $E$ such that $E\not=\pi_M(\mathcal{E})$.  

\begin{lem}\label{complete graph}
If $G$ is a complete graph, then given a maximal invariant chain transitive set $\mathcal{E}\subset M\times\Delta$, $\mathcal{E}=\ell(\pi_M\mathcal{E})$.
\end{lem}
\begin{proof}
It remains to show that $\ell(\pi_M\mathcal{E})\subset\mathcal{E}$; this can be done by showing that $\ell(\pi_M\mathcal{E})$ is chain transitive.  \\
\indent Let $(x,f),(y,g)\in \mathcal{E}$ and pick $\varepsilon, T>0$.  Recall that 

$$d(f,g)=\displaystyle\sum_{i=-\infty}^{\infty}\left(\frac{1}{h}\displaystyle\int_{ih}^{(i+1)h}\delta(f,g,t)dt\right)*4^{-|i|}$$
 where
 $$\delta(f,g,t)=\left\{
     \begin{array}{lr}
       1 &  f(t)\not=g(t)\\
       0 &  f(t)=g(t)
     \end{array}
   \right. .$$

There exists $N\in\mathbb{N}$ such that 
$$2\displaystyle\sum_{i=N}^{\infty}\frac{1}{4^{|i|}}<\varepsilon/2.$$
Chain transitivity of $\mathcal{E}$ means there exists $k\in\mathbb{N}$ and $x_0,\ldots,x_k\in M$, $f_0,\ldots,f_{k-1}\in\overline{\Delta}$, $t_0,\ldots,t_{k-1}>T$ with $x_0=\varphi(2T,x,f)$ and $x_k=\varphi(-T,y,g)$ with 
$$d(\varphi(t_j,x_j,f_j),x_{j+1})<\varepsilon.$$
Without loss of generality, let $t>1$.  Then by Theorem \ref{h flow}, we can set
$$t_0=\cdots=t_{k-1}=h.$$
Define
$$t_{-2}=Nh, \,\,\,x_{-2}=x, \,\,\,g_{-2}=f$$
$$t_{-1}=t,x_{-1}=\varphi(Nh,x,f), g_{-1}= \left\{
     \begin{array}{lr}
       f(t_{-2}+t) &  t\leq t_1\\
       f_0(t-t_{-1}) &  t>t_1
     \end{array}
   \right.$$

Let $t_0,\,\,\,\ldots,\,\,\,t_{k-1}$ and $x_0,\ldots,x_k$ be given as before, and let 
$$t_k=Nh, x_{k+1}=y,g_{k+1}=g.$$
Now, for $j=0,\ldots,k-2$ we define
$$g_j(t)=\left\{
     \begin{array}{lr}
      g_{j-1}(t_{j-1}+t) &  t\leq 0\\
       f_j(t) & 0<t\leq t_j\\
       f_{j+1}(t-t_j) & t>t_j
    
     \end{array}
   \right.$$
   $$g_{k-1}=\left\{
     \begin{array}{lr}
       g_{k-2}(t_{k-2}+t) & t\leq0\\
       f_{k-1}(t)&  0<t\leq t_{k-1}\\
       g(t-t_{k-1}-Nh) & t>t_{k-1}
     \end{array}
   \right.$$
   
   $$g_k=\left\{
     \begin{array}{lr}
       g_{k-1}(t_{k-1}+t) &  t\leq 0\\
       g(t-Nh) &  t>0
     \end{array}
   \right.$$

We then claim that, by construction, all $g_j$'s are elements of $\Delta$.  Recall that by Definition \ref{delta} functions in $\Delta$ require that 
$$\{f(ih)\}_{i\in\mathbb{Z}}\in\Omega$$ for all $f\in\Delta$. Since the graph $G$ is complete, clearly for each $f_i$ the jumps between vertices are allowed by the graph.  The ``stitching" together of pieces of the functions $f_i$'s is also allowed by the graph $G$ associated $\Delta$ because of the completeness of $G$. \\
We further require that the functions $g_i$ be piecewise constant on intervals of length $h$.  By setting $t_j=Nh$ for all $j\in\{-2,-1,\ldots,k-1\}$ this property is satisfied as well.  Thus, $g_j\in\Delta$ for all $j$.
 \\

\indent We further claim that for all $j=-2,-1,\ldots,k$,
$$d(g_j(t_j+\cdot),g_{j+1})<\varepsilon.$$

By choice of $N$, on has that for all $d_1,d_2\in\Delta$

\begin{eqnarray*}
d(d_1,d_2)&=& \displaystyle\sum_{i=-\infty}^\infty \left(\frac{1}{h}\displaystyle\int_{ih}^{(i+1)h}\delta(d_1,d_2,t)dt\right)*4^{-|i|}\\
&\leq& \displaystyle\sum_{i=-N}^N\left[ \left(\frac{1}{h}\displaystyle\int_{ih}^{(i+1)h}\delta(d_1,d_2,t)dt\right)*4^{-|i|}\right] +\varepsilon/2\\
\end{eqnarray*}
Thus it suffices to show that for the considered functions, the integrands vanish.  Notice by definition, for all $i\in\{-2,-1,\ldots,k-1\}$,
$g_i(t+Nh)=g_{i+1}$ for all $-Nh<t<Nh$.  Thus, 
$\delta(g_i(t+Nh),g_{i+1}(t),t)=0$ for all $-Nh<t<Nh$, and therefore,
$$\displaystyle\int_{ih}^{(i+1)h}\delta(f,g,t)dt$$ for all $i\in\{-N,\ldots,N-1\}$.
Thus for all $j=-2,-1,\ldots,k$,
$$d(g_j(t_j+\cdot),d_{j+1})<\varepsilon.$$
\end{proof}

\bibliographystyle{siam}
\bibliography{e,TRecurrence.bib}

\end{document}